\documentclass[12pt]{amsart}
\usepackage{datetime}

\makeatletter
\@namedef{subjclassname@2010}{%
  \textup{2010} Mathematics Subject Classification}
\makeatother

\usepackage{comment}
\usepackage{color}
\usepackage{amsmath}
\usepackage{amsthm}
\usepackage{amssymb}
\usepackage{graphicx}
\usepackage{enumerate}
\usepackage{amsfonts}
\usepackage{diagrams}
\usepackage{mathrsfs}
\usepackage{parskip}
\usepackage{mathdots}
\usepackage[normalem]{ulem}
\usepackage{color}

\newcommand{\tcb}[1]{\textcolor{blue}{#1}}
\numberwithin{equation}{section}
\theoremstyle{plain}
\newtheorem{Proposition}[equation]{Proposition}
\newtheorem{Corollary}[equation]{Corollary}
\newtheorem*{Corollary*}{Corollary}
\newtheorem{Theorem}[equation]{Theorem}
\newtheorem*{Theorem*}{Theorem}
\newtheorem{Lemma}[equation]{Lemma}
\theoremstyle{definition}

\newtheorem{Example}[equation]{Example}
\newtheorem{Remark}[equation]{Remark}

\allowdisplaybreaks

\def\HH{\mathscr{H}}
\def\MM{\mathscr{M}}

\def\C{\mathbb{C}}
\def\R{\mathbb{R}}
\def\D{\mathbb{D}}
\def\T{\mathbb{T}}
\def\N{\mathbb{N}}
\def\Z{\mathbb{Z}}

\def\K{\mathcal{K}}
\def\phi{\varphi}
\def\mU{\mathcal{U}}

\newcommand{\dist}{\operatorname{dist}}
\renewcommand{\ker}{\operatorname{Ker}}
\renewcommand{\dim}{\operatorname{dim}}

\newcommand{\beqa}{\begin{eqnarray*}}
\newcommand{\eeqa}{\end{eqnarray*}}
\newcommand{\dst}{\displaystyle}

\renewcommand{\leq}{\leqslant}
\renewcommand{\subset}{\subseteq}

\title{Multipliers between model spaces}

\author[Fricain]{Emmanuel Fricain}
 \address{Laboratoire Paul Painlev\'e, Universit\'e Lille 1, 59 655 Villeneuve d'Ascq C\'edex }
 \email{emmanuel.fricain@math.univ-lille1.fr}

\author[Hartmann]{Andreas Hartmann}
\address{Institut de Math\'ematiques de Bordeaux, Universit\'e de Bordeaux, 351 cours de la Lib\'eration 33405 Talence C\'edex, France}
\email{Andreas.Hartmann@math.u-bordeaux.fr}

\author[Ross]{William T. Ross}
	\address{Department of Mathematics and Computer Science, University of Richmond, Richmond, VA 23173, USA}
	\email{wross@richmond.edu}

\date{}



\begin{document}

\baselineskip=17pt

\begin{abstract}In this paper we examine the multipliers from one model space to another.  
\end{abstract}

\maketitle

\section{Introduction}

For an inner function $\Theta$, let $\K_{\Theta} := H^2 \cap (\Theta H^2)^{\perp}$
 denote the {\em model space} of the open unit disk $\D$ corresponding to $\Theta$. In this paper, we explore, for a pair of inner functions $u$ and $v$, the {\em multipliers} 
$$
\MM(u, v) := \{\phi \in \operatorname{Hol}(\D): \phi \K_u \subset \K_v\}
$$
between $\K_u$ and $\K_v$. 

One motivation for this paper comes from the work of Crofoot \cite{Crofoot} who considered a more restricted version of $\MM(u, v)$ namely 
$
\{\phi \in \operatorname{Hol}(\D): \phi \K_u = \K_v\},
$
in other words, the multipliers from $\K_u$ {\em onto} $\K_v$ (see also \cite[Def.~3.7]{MR3460059}). As it turns out, these onto multipliers are
unique up to multiplicative constants and are outer functions. Another motivation comes from examining pre-orders on partial isometries \cite{GMR, MR3616198}. 

The Crofoot discussion becomes quite different if we relax the (onto) multiplier condition $\phi \K_{u}  = \K_v$ to just
$\phi \K_u \subset \K_v$. For one, as we shall see below, these (into but not necessarily onto) multipliers need not be outer functions. 
Secondly, unlike the onto multipliers, the into multipliers need not be unique. In fact, we give an example of when $\MM(u, v)$ is infinite dimensional and contains unbounded functions. 

After a few initial observations about $\MM(u, v)$ we will reformulate the description of $\MM(u, v)$ in terms of Carleson measures of model spaces and kernels of Toeplitz operators. Along the way, we will describe $\MM(u, v)$ when $v$ is an inner multiple of $u$. We will then relate $\MM(u, v)$ to the boundary spectra of $u$ and $v$ along with their sub-level sets.

We also consider multipliers for the model spaces of the upper half plane. In this setting we discuss a particular entire function introduced by Lyubarskii and Seip which allows us to deduce the existence of unbounded onto multipliers connecting to a question raised by Crofoot. As discussed earlier, the onto multipliers are unique (up to multiplicative constants) and thus the multipliers algebra in this case is one dimensional. In the spirit of the  Lyubarskii and Seip construction above, we produce $u$ and $v$ such that $\MM(u, v) = \C \phi$, yet $\phi$ is not an onto multiplier. 

\subsection*{Acknowledgement}
We would like to thank Cristina C\^amara for some insightful discussions along with the anonymous referee for some suggested improvements (especially Theorem \ref{ppppooiuy},  Theorem \ref{|22122144141}, Proposition \ref{oooooAC}, and suggesting the problem we answered in Theorem \ref{39782194uro3iq}).

\section{Notation, observations, and simplifications}\label{MM}

We assume the reader is familiar with the Hardy space $H^2$ \cite{Duren, Garnett} and model spaces $\K_u$  \cite{MSGMR, Nik}.  In this paper, $\D$ is the open unit disk, $\T$ the unit circle, $m$ normalized Lebesgue measure on $\T$, and $L^2$ the standard Lebesgue space $L^2 := L^2(\T, m)$ with norm $\|f\|$
 and inner product $\langle \cdot, \cdot \rangle$. The bounded analytic functions on $\D$ are denoted by $H^{\infty}$.
Recall that $H^2$ is a reproducing kernel Hilbert space with kernel $k_{\lambda}(z)=(1-\overline{\lambda}z)^{-1}$. 

We begin with some useful observations. 
First notice that
 $\MM(u,v)\subseteq H^2$.
Indeed, if 
$$
 k_{\lambda}^u(z)=\frac{1-\overline{u(\lambda)}u(z)}{1-\overline{\lambda}z},\quad
 \lambda,z\in\D,
$$
denotes the reproducing kernel for $\K_u$,
then $k^{u}_0 = 1 - \overline{u(0)} u \in \K_u$ is an invertible element of $H^{\infty}$. Thus if $\phi \in \MM(u, v)$ then $\phi k^{u}_0 \in \K_v \subseteq H^2$ from which the result follows.

Furthermore, when $\phi \in \MM(u, v)$, the closed graph theorem says that $M_{\phi} f = \phi f$ is a bounded operator from $\K_u$ to $\K_v$ and standard arguments show that
$
M_{\phi}^{*} k_{\lambda}^{v} = \overline{\phi(\lambda)} k_{\lambda}^{u}.
$
Since 
$$\|k_{\lambda}^{u}\|^2 =  \frac{1 - |u(\lambda)|^2}{1 - |\lambda|^2},$$ it follows that 
\begin{equation}\label{8whhsuyshffsdf}
|\phi(\lambda)|^2 (1 - |u(\lambda)|^2) \lesssim (1 - |v(\lambda)|^2), \quad \lambda \in \D.
\end{equation}
Though this inequality will be used later on, it does not prove that $\phi$ is always bounded. The following Proposition summarizes some basic facts which follow, or can be gleaned, from Crofoot's  paper \cite{Crofoot}. 

\begin{Proposition}\label{sdfjsldkfjldskf66666}
Let $u$ and $v$ be inner functions.
\begin{enumerate}
\item[(i)] $\MM(u, u) = \C$.
\item[(ii)] If $\phi \K_u = \K_v$ then $\phi$ is outer. 
\item[(iii)] $\C \subset \MM(u, v)$ if and only if $u$ divides $v$. 
\item[(iv)] Suppose $u$ divides $v$ and $u$ is not a constant multiple of $v$. Then $\MM(v, u) = \{0\}$.
\item[(v)] If $\phi \in \MM(u, v)$ and $F$ is the outer factor of $\phi$, then $F \in \MM(u, v)$. 
\item[(vi)] If $a \in \D$ and 
$
u_a := \frac{u - a}{1 - \overline{a} u},$
then 
${\displaystyle \frac{1}{1 - \overline{a} u} \K_u  = \K_{u_a}}.$
\end{enumerate}
\end{Proposition}

The map $f \mapsto (1 - \overline{a} u)^{-1} f$ from $\K_{u}$ onto $\K_{u_a}$ is a constant multiple of the unitary {\em Crofoot transform}. Using operator theory techniques, Crofoot \cite[Theorem 14]{Crofoot} showed that when the space of onto multipliers is non-empty, then $\sigma(u) = \sigma(v)$, where 
$$\sigma(u) := \Big\{\xi \in \T: \varliminf_{z \to \xi} |u(z)| = 0\Big\}$$ is the {\em boundary spectrum} of an inner function. 
The following result is the $\MM(u, v)$ analogue of this  where our proof uses function theory.

\begin{Proposition}\label{23094ygtoriefb}
If $\MM(u, v) \not = \{0\}$ then $\sigma(u) \subseteq \sigma(v)$. 
\end{Proposition}

\begin{proof}
Without loss of generality, we can use Proposition \ref{sdfjsldkfjldskf66666} (vi) and assume that $u(0) = 0$ (the Crofoot transform preserves the regular points in $\T$). Then $1 \in \K_u$ and so 
$\phi \K_u \subseteq \K_v \implies \phi \in \K_v.$
Pick $\zeta\in \T\setminus \sigma(v)$ (a regular point for $v$).
Then \cite[p.~153]{MSGMR} every function in $\K_v$ has an analytic continuation to a two-dimensional open neighborhood $\Omega$ of $\zeta$. 
In particular, $\phi \in \K_v$ enjoys this property. 
For every $f \in \K_u$, $g:=\varphi f\in \K_v$ has an analytic continuation to $\Omega$ and so $f=g/\varphi$ is either analytic on $\Omega$ or has a pole of order at least
$1$ at $\zeta$. But this second case is not possible since $f\in H^2$ must be square integrable on 
$\T$. Hence $f$ extends analytically to $\Omega$ and thus $\zeta\in \T
\setminus \sigma(u)$.
\end{proof}


\section{A useful reformulation}

In this section we reformulate the description of $\MM(u, v)$
in terms of kernels of Toeplitz operators and Carleson measures for model spaces.

\begin{Theorem}\label{ppppooiuy}
For inner $u$ and $v$ and $\phi \in H^2$, the following are equivalent: 
\begin{enumerate}
\item[(i)] $\phi \in \MM(u, v)$;
\item[(ii)] $\phi S^*u \in \K_{v}$ and $|\phi|^2 dm$ is a Carleson measure for $\K_u$,  i.e., 
$$\int_{\T} |f|^2 |\phi|^2 dm \lesssim \|f\|^2, \quad f \in \K_u;$$
\item[(iii)] $\phi \in \ker T_{\overline{z v}u}$ and $|\phi|^2 dm$ is a Carleson measure for $\K_u$, where $T_{\overline{z v}u} f = P_{+}(\overline{z v} u f)$ is the standard Toeplitz operator on $H^2$.
\end{enumerate}
Furthermore, the following are equivalent:
\begin{enumerate}
\item[(iv)] $\phi \in \MM(u, v) \cap H^{\infty}$;
\item[(v)] $ \phi S^*u \in \K_{v} \cap H^{\infty}$.
\item[(vi)] $\phi \in \ker T_{\overline{z v}u}\cap H^{\infty}$.
\end{enumerate}
\end{Theorem}

\begin{proof}
Recall that $\ker T_{\overline{u}}=\K_u$ \cite[p.~108]{MSGMR} and that $T_{f g} = T_{f} T_{g}$ if either $\overline{f} \in H^{\infty}$ or $g \in H^{\infty}$ \cite[p.~97]{MSGMR}. Also observe that $T_{\overline{z}} = S^{*}$ and that $T_{1 - u(0) \overline{u}}$ is invertible. Using these facts, along with the identity (on $\T$), 
\begin{equation}\label{66sd11qaswqa}
\varphi S^*u=\varphi \overline{z}(u-u(0))=\varphi \overline{z} u (1 - u(0) \overline{u}),
\end{equation}
 it follows that 
$\phi S^{*} u \in \K_v \iff \phi \in \ker  T_{\overline{z v}u}.$
This yields $(ii) \iff (iii)$ and $(vi) \implies (v)$.
The implication $(v) \implies (vi)$ needs an additional argument. Indeed, suppose that
$ \phi S^*u \in \K_{v} \cap H^{\infty}$. Then the above equivalences yield 
$\phi \in \ker T_{\overline{z v}u}$, and we just have to check that $\phi$ is bounded.
We already know that $\phi\in H^2$. Thus  in order to verify $\phi \in H^{\infty}$, it suffices to
prove that $\phi|_{\T} \in L^{\infty}$ (Smirnov's theorem \cite[p.~28]{Duren}). By assumption, $\varphi S^*u=g\in H^{\infty}$ and \eqref{66sd11qaswqa} shows that $\phi|_{\T} \in L^{\infty}$. 

The implications $(i) \implies (ii)$ and  $(iv) \implies (v)$ are automatic. The implication $(v) \implies (iv)$ becomes automatic once we have shown 
$(ii) \implies (i)$.  So it remains to prove $(ii) \implies  (i)$. Observe that $f \in \K_v$ if and only if $v \overline{z f} \in \K_v$ (see \cite[p.~105]{MSGMR}). We know that $\varphi S^*u\in K_v$ which means, via \eqref{66sd11qaswqa} that $v\overline{u\varphi}\in H^2$. Since $|\phi|^2 dm$ is a Carleson measure for $\K_u$ (i.e., $\phi \in \MM(\K_u, H^2)$) it suffices to show that $\phi g \in \K_v$ for all $g \in \K_u \cap H^{\infty}$ (which is dense in $\K_u$). Indeed, 
$$v \overline{z \phi g} = v \overline{u} \overline{\phi} \cdot u \overline{z} \overline{g} \in H^2 \cdot H^{\infty} \subset H^2. \qedhere$$
\end{proof}

\begin{Corollary}\label{MMcont} 
$ \ker{T_{\overline{z v} u}} \cap H^{\infty} = \MM(u, v) \cap H^{\infty} \subset \MM(u,v)\subseteq \ker T_{\overline{z v}u}$.
\end{Corollary}

We will see in Example \ref{1551618714499} below that, 
in general, $\MM(u,v) \subsetneq \ker T_{\overline{z v}u}$.

\begin{Corollary}\label{7ujmpokjn}
Suppose $u$ and $v$ are inner and $v = u I$. Then the following are equivalent:
\begin{enumerate}
\item[(i)] $\phi \in \MM(u, v)$;
\item[(ii)] $\phi \in \K_{z I}$ and $|\phi|^2 dm$ is a Carleson measure for $\K_u$.
\end{enumerate}
Furthermore, the following are equivalent: 
\begin{enumerate}
\item[(iii)] $\phi \in \MM(u, v) \cap H^{\infty}$;
\item[(iv)] $\phi \in \K_{z I} \cap H^{\infty}$.
\end{enumerate}
If $I$ is a finite Blaschke product then $\MM(u, v) \cap H^{\infty} = \MM(u, v) = \K_{z I}$.
\end{Corollary}

Our next result uses analytic continuation and the boundary spectrum to construct a class of inner functions $u$ and $v$, with $v = u I$, such that the Carleson condition on $|\phi|^2 dm$ is automatic as soon as $\phi \in \K_{z I}$. 

\begin{Theorem}\label{11sdfsdhjf}
Let $u$ and $v$ be inner functions with and $v = u I$ for some inner function $I$. Suppose further that $\sigma(u) \cap \sigma(I) = \varnothing$. Then $\MM(u, v) = \K_{z I}$. Furthermore, if $I$ is not a finite Blaschke product then $\MM(u, v)$ contains unbounded functions. 
\end{Theorem}

\begin{proof}
By Corollary \ref{7ujmpokjn}, we just need to check that $|\phi|^2 dm$ is a Carleson measure for $\K_u$ for every $\phi \in \K_{z I}$.  Let $V$ be a two dimensional neighborhood of $\sigma(I)$ that is far from $\sigma(u)$. By \cite[p.~153]{MSGMR} $\phi$ extends analytically outside $V$ (i.e., $\overline \D\setminus V$) and thus can be assumed to be bounded outside $V$. Similarly, every $f \in \K_u$ extends analytically to $V$ and can be assumed to be bounded there. From here it follows that $\phi f \in H^2$. 
By the Closed Graph Theorem, $\phi \in \MM(\K_u, H^2)$, equivalently, $|\phi|^2 dm$ is a Carleson measure for $\K_u$. 

For the last part, note that if $I$ is not a finite Blaschke product then $\K_{z I}$ is infinite dimensional \cite[p.~117]{MSGMR} and thus, via a well-known theorem of Grothendieck \cite{Rudin-FA},  contains unbounded functions. 
\end{proof}

We now construct an example of when $\ker T_{zI}=\ker T_{\overline{zv}u}$ contains 
functions which do not define Carleson measures for $\K_u$ and thus 
$\MM(u, v) \subsetneq \ker{T_{\overline{z v} u}}$. Hence the Carleson condition is important in Theorem \ref{ppppooiuy}.


\begin{Example}\label{1551618714499}
Set $\lambda_n = 1 - 2^{-n}$, $n \geqslant 1$, and note this is the zero sequence of an interpolating Blaschke product $I$. With $w_n = n^{-1}$, notice that $\sum_{n \geqslant 1} w_{n}^{2} < \infty$. By an interpolation theorem from \cite[p.~135]{Nik}, there is a $\phi \in \K_{I} \subseteq \K_{z I}$ such that 
$$
\phi(\lambda_n) = \frac{w_n}{(1 - |\lambda_n|^2)^{1/2}} \asymp \frac{2^{n/2}}{n} \to \infty.
$$ Now take $u(z) = \exp((z+1)/(z-1))$ and observe that since $\lambda_n \to 1$ on $(0, 1)$ we have 
$u(\lambda_n) \to 0$.
If $v = u I$ then
$\phi \in \K_I \subseteq \K_{z I} = \ker{T_{\overline{z v} u}}.$
However, $\phi \not \in \MM(u, v)$ since, if it were, \eqref{8whhsuyshffsdf} would imply that
$$|\phi(\lambda_n)|^2 (1 - |u(\lambda_n)|^2) \lesssim 1 - |v(\lambda_n)|^2 \lesssim 1.$$
The above discussion now yields a contradiction. Thus we have 
$\MM(u, v) \subsetneq \ker{T_{\overline{z v} u}} = \K_{z I}$.
\end{Example}
 
 \section{Finite dimensional case}

 We now consider finite dimensional model spaces.
For an inner $u$, the {\em degree} of $u$ is $n$ if $u$ is a finite Blaschke product with $n$ zeros  and equal to $\infty$ otherwise.
When $u$ is a finite Blaschke product with $n$ zeros $\{\lambda_1, \cdots, \lambda_n\}$, we have
\begin{equation}\label{sdkfhsdf53e0}
\K_u = \Big\{\frac{p(z)}{\prod_{j = 1}^{n} (1 - \overline{\lambda_j} z)}: p \in \mathcal{P}_{n - 1}\Big\},
\end{equation}
where $\mathcal{P}_{n - 1}$ are the polynomials of degree at most $n - 1$. 

\begin{Theorem}\label{7786ettetete}
If $u$ is a finite Blaschke product with zeros $\{a_1, \ldots, a_m\}$ and $v$ is a finite Blaschke product with zeros $\{b_1, \ldots, b_n\}$ where $m \leqslant n$, and the zeros are repeated according to their multiplicity, then 
$$\MM(u, v) = \MM(u, v) \cap H^{\infty} = \Big\{q(z) \frac{\prod_{i = 1}^{m} (1 - \overline{a_i} z)}{\prod_{j = 1}^{n} (1 - \overline{b_j} z)}: q \in \mathcal{P}_{n -  m}\Big\}.$$
\end{Theorem}

\begin{proof}
The  $\supseteq$ containment follows essentially from \eqref{sdkfhsdf53e0}.
For the $\subseteq$ containment, notice from Theorem \ref{ppppooiuy} that $\phi \in \MM(u, v) \implies \phi \in \ker{T_{\overline{z v} u}}$ which is equivalent to
 $$u \phi \in \ker{T_{\overline{z v}} = \K_{z v}} = \Big\{\frac{p(z)}{\prod_{j = 1}^{n} (1 - \overline{b_j} z)}: p \in \mathcal{P}_{n}\Big\}\subseteq H^{\infty}.$$ The result now follows. 
\end{proof}

\begin{Theorem}\label{thm3.21}
If $u$ is a finite Blaschke product and $v$ is any inner function with infinite degree, then $\MM(u, v) \cap H^{\infty} \not = \{0\}$.
\end{Theorem}

\begin{proof}
By \cite[p.~75]{Garnett} there is an $a \in \D$ (in fact ``most'' $a$) such that the Frostman shift $v_a = \frac{v - a}{1 - \overline{a} v}$ of $v$  is a Blaschke product of infinite degree. Factor $v_a = I J$, where $I$ and $J$ are Blaschke products with the degree of $I$ equal to the degree of $u$, and use 
\cite[p.~14]{Nik} to obtain $\K_I \subset \K_{v_a}$.
From Theorem \ref{7786ettetete} there is a rational $\phi \in H^{\infty}$ such that 
$\phi \K_u \subset \K_I \subset \K_{v_a}$. Proposition \ref{sdfjsldkfjldskf66666} (vi) now yields
$(1 - \overline{a} v) \phi \K_u \subset \K_v$.
\end{proof}


\section{Sub-level sets}\label{SubLev}

In this section we discuss some results using sub-level sets of inner functions.
We start with a  ``maximum principle'' result of Cohn \cite{MR959220}.

\begin{Theorem}\label{dfdfsdfsdfk}
Suppose $\Theta$ is inner and $f \in \K_{\Theta}$  is bounded on $\{|\Theta| < \epsilon\}$ for some $\epsilon \in (0, 1)$. Then $f \in H^{\infty}.$
\end{Theorem}

This result can be used to show that under certain circumstances, all multipliers must be bounded. 

\begin{Corollary}\label{CbWtW}
Let $u$ and $v$ be inner. If, for some $\epsilon_1, \epsilon_2 \in (0, 1)$, $\{|v| < \epsilon_2\} \subseteq \{|u| < \epsilon_1\}$, then 
$\MM(u, v) = \ker T_{\overline{z v} u} \cap H^{\infty}$.
\end{Corollary}

\begin{proof}
Let $\phi \in \MM(u, v)$.  The estimate in \eqref{8whhsuyshffsdf} says that when $\lambda \in \{|v| < \epsilon_2\} \subseteq \{|u| < \epsilon_1\}$ we have
$|\phi(\lambda)|^2 \lesssim (1 - \epsilon_{1}^{2})^{-1}$ and thus $\phi$ is bounded on $\{|v| < \epsilon_2\}$. Since $k^{u}_0 = 1 - \overline{u(0)} u \in \K_u$ and bounded on $\D$, we see that $k^{u}_0 \phi \in \K_v$ and bounded on $\{|v| < \epsilon_2\}$. Apply Theorem \ref{dfdfsdfsdfk} to obtain $k^{u}_0 \phi \in H^{\infty}$. Since $k^{u}_0$ is invertible in $H^{\infty}$, we get $\phi \in H^{\infty}$. Now apply Corollary \ref{MMcont}. 
\end{proof}

\begin{Example}\label{example-inner-function-multiplier-bouded}
Let $u$ be any singular inner function and $v = u^{\alpha}$ for some $\alpha > 1$ (or perhaps $u$ a Blaschke product, or any inner function, and $\alpha \in \N$). Notice that $u$ divides $v$ and so $\MM(u, v) \not = \{0\}$ (Corollary \ref{7ujmpokjn}). 
Furthermore if $\epsilon_2 \in (0, 1)$ and $z \in \{|v| < \epsilon_2\}$ then 
$|u(z)|^{1/\alpha} \leq \epsilon_{2}^{1/\alpha}.$
Setting $\epsilon_1 = \epsilon_{2}^{1/\alpha}$ we see that $\{|v| < \epsilon_2\} \subset \{|u| < \epsilon_1\}$.
Corollary \ref{CbWtW} yields $\MM(u, v) \subseteq H^{\infty}$. Combine this with Corollary \ref{7ujmpokjn} to see that
$\MM(u, v) = \K_{z u^{\alpha - 1}} \cap H^{\infty}.$ \end{Example}

Carleson measure results of Cohn \cite{MR705235, MR816238} allow us, in the special case where $u$ satisfies the connected level set condition (i.e., $\{|u| < \varepsilon\}$ is connected for some $\varepsilon > 0$), to replace the condition that $|\phi|^2 dm$ is a Carleson measure in Theorem \ref{ppppooiuy} and Corollary \ref{7ujmpokjn} with 
$$
\sup_{\lambda \in \D} (1 - |u(\lambda)|^2)  \int_{\T} \frac{1 - |\lambda|^2}{|\xi - \lambda|^2} |\phi(\xi)|^2 dm(\xi) < \infty.
$$

\section{The upper-half plane}\label{Sec-UHP}\label{seven}

We will now turn to the upper half plane which in certain situations is a
more appropriate setting. If $\C_{+}$ denotes the upper-half plane, we set 
$\mathscr{H}^{2} $
to be the corresponding Hardy space. 
There is a natural unitary operator $\mathcal{U}$  from $H^2$ onto $\mathscr{H}^2$ given by 
$$(\mathcal{U} f)(z) := \frac{1}{\sqrt{\pi} (z + i)} f(\omega(z)),$$
where 
$
\omega(z) := \frac{z - i}{z + i}
$
 maps $\C_{+}$ onto $\D$ and  $\R \cup\{-\infty, \infty\}$ onto $\T$. 
As with $H^2$, one can define, for $\Psi \in L^{\infty}(\R)$, the Toeplitz operator 
$T_{\Psi}$ on $\mathscr{H}^2$. 

For an inner function $U$ on $\C_{+}$,
we define the model space 
$$
\mathscr{K}_{U} := \mathscr{H}^2 \cap (U \mathscr{H}^2)^{\perp}
.$$
 The corresponding reproducing kernel function for $\mathscr{K}_{U}$ is 
$$K_{\lambda}^{U}(z) := \frac{i}{2 \pi} \frac{1 - \overline{U(\lambda)} U(z)}{z - \overline{\lambda}}, \quad \lambda, z \in \C_{+}.$$ Note that if $u$ is an inner function on $\D$ and $U = u \circ \omega$, then $U$ is an inner function on  $\C_{+}$ (and vice versa). Furthermore,  
$
 \mathcal{U} \K_u = \mathscr{K}_{U}.
$

\subsection*{Multipliers and kernels} 

In this subsection we need the elementary Blaschke factor on $\C_{+}$ with zero at $i$: 
$$b_{i}^{+}(z) := \frac{z - i}{z + i},$$
and 
$$k_i(z)=\frac{1}{\sqrt{\pi}}\frac{1}{z+i},$$
the corresponding  kernel at $i$. Observe that $\mU f=k_i\times (f\circ\omega), f\in H^2$.

We begin with some elementary but useful facts. The proofs are straightforward. 




\begin{Lemma}\label{L1ker}
Let $\psi\in L^{\infty}(\T)$ and $\Psi=\psi\circ \omega$.
Then $$f\in \ker T_{\psi} \iff F:=\mathcal{U} f
\in \ker T_{\Psi}.$$
\end{Lemma}

\begin{Lemma}\label{L2mult}
$\varphi\in \MM(u,v)$ if and only if $\Phi=\varphi\circ\omega\in \MM(U,V)$.
\end{Lemma}

\begin{Corollary}\label{55srs6s7sancvc}
With the notation from above, the following are equivalent for $\Phi$ analytic on $\mathbb C_+$:
\begin{enumerate}
\item[(i)] $\Phi \in \MM(U, V)$;
\item[(ii)] $\Phi k_i\in \ker T_{\overline{b_{i}^{+}V} U}$ and  $|\Phi|^2dx$ is a Carleson
measure for $\mathscr{K}_U$.
\end{enumerate}
\end{Corollary}

We now discuss a situation when the Carleson condition becomes more tractable. We begin with a result from Baranov \cite[Thm.~5.1]{MR1784683}. 

\begin{Theorem}\label{Lem:caracterisation-carleson-measure}
Let $U$ be an inner function in $\mathbb C_+$ such that $|U'(x)|\asymp 1$, $x\in\R$. For a positive Borel measure $\mu$ on $\mathbb R$, the following are equivalent:
\begin{enumerate}[(i)]
\item $\mu$ is a Carleson measure for $\mathscr{K}_U$.
\item We have
$
{\displaystyle M:=\sup_{x\in\mathbb R}\mu([x,x+1])<\infty.}
$
\end{enumerate} 
\end{Theorem}


\begin{Remark}\label{Dya}
Concerning boundedness of $U'$, Dyakonov \cite{MR1097956} proved that the following conditions are equivalent: 
\begin{enumerate}[(i)]
\item $U'\in L^\infty(\mathbb R)$.
\item For some $\eta, \varepsilon>0$, $\{z\in \C_+:|U(z)|<\epsilon\} \subset \{z\in \C:\Im z>\eta\}.$
\item $\mathscr{K}_U\subset \mathscr{H}^\infty$, the bounded analytic functions on $\C_{+}$.
\end{enumerate}

\end{Remark}

\begin{Theorem}\label{|22122144141}
Let $U$ and $V$ be inner functions with $|U'(x)| \asymp 1, x \in \R$. 
Then 
\[
 \MM(U,V) =  \Big\{\Phi\in (z+i)\ker T_{\overline{b_{i}^{+}V} U}: M:=\sup_{x\in\R}\int_x^{x+1}|\Phi(t)|^2dt
 <\infty\Big\}.
\]
\end{Theorem}

\begin{proof}
Observe that 
$\Phi k_i\in \ker T_{\overline{b_{i}^{+}V} U} \iff \Phi\in (z+i)\ker
T_{\overline{b_{i}^{+}V} U}$
and apply Corollary~\ref{55srs6s7sancvc} and Theorem~\ref{Lem:caracterisation-carleson-measure}.
%
\end{proof}

\begin{Lemma}
We have 
$$F\in \ker T_{\overline{V} U} \iff F\in \Big((z+i)\ker
T_{\overline{b_{i}^{+}V} U}\Big)\cap \mathscr{H}^2.$$
\end{Lemma}

\begin{proof}
The function $F$ belongs to $\ker T_{U\overline{V}}$ if and only if there is a $\psi\in \mathscr{H}^2$ such that
$
 \overline{V} UF=\overline{\psi}.
$
A calculation shows that 
$$
\overline{V(x)b_{i}^{+}(x)}U(x)  F(x) k_i(x) =\overline{(\psi k_i)(x)}, \quad x \in \R.
$$
Hence $Fk_i\in \ker T_{\overline{b_{i}^{+}V} U}$ and so $F\in (z+i)\ker T_{\overline{b_{i}^{+}V} U}$. 
\\
\\
The converse argument is in the same spirit. Indeed, when 
 $$F\in (z+i)\ker
T_{\overline{b_{i}^{+}V} U}\cap \mathscr{H}^2,$$
we get
$
 F(x)(\overline{V(x)} U(x))=\overline{\psi(x)}(x-i)=\overline{\psi(x)(x+i)}.
$
Since $F\in \mathscr{H}^2$, and $U\overline{V}$ is bounded, we deduce that $\psi (z+i)\in \mathscr{H}^2$, and so
$F\in \ker T_{\overline{V} U}$.
\end{proof}



\begin{Corollary}\label{charactmultbddder}
Let $U$ and $V$ be inner functions with $|U'(x)| \asymp 1, x \in \R$. 
Then
$
 \MM(U,V)\cap \mathscr{H}^2=\ker T_{U\overline{V}}.
$
\end{Corollary}

\subsection*{Resolving a question of Crofoot} 
We notice that an example constructed in \cite{Lyubarski-Seip} answers a question of Crofoot \cite[p.~244]{Crofoot}.
We will state the result for the model spaces $\mathscr{K}_{\Theta}$ of the upper-half plane and then use Lemma \ref{L2mult}.

\begin{Theorem}\label{nnxbbxvxrx55}
There are two inner functions $B$ and $\Theta$ on $\C_{+}$ and an unbounded analytic function $\Psi$ on $\C_{+}$ such that $\Psi\K_B = \K_\Theta$.
\end{Theorem}

 The construction is based on the relationship between the model subspaces generated by meromorphic inner functions and the de Branges spaces of entire functions \cite{MR0229011}. 

First we define the {\em Paley-Wiener} class 
$$
PW =  \Big\{F \in \operatorname{Hol}(\C): \frac{F}{e^{- i \pi z}}, \frac{F^{*}}{e^{- i \pi z}} \in \mathscr{H}^2\Big\}, \quad F^{*}(z) := \overline{F(\overline{z})}.$$

 Let $E$ be an entire function which belongs to the {\em Hermite--Biehler class} $H\!B$, i.e., 
$$|E(z)| \geqslant |E(\overline{z})|, \quad  \Im z > 0$$ and 
$E$ does not have any zeros in $\C_{+}^{-}$ (the closed upper half plane). With $E \in H\!B$, define the {\em de Branges space}
\begin{equation}\label{o48tugjflkds}
\HH(E) := \Big\{F \in \operatorname{Hol}(\C): \frac{F}{E}, \frac{F^{*}}{E} \in \mathscr{H}^2\Big\}.
\end{equation}
The norm in $\HH(E)$ is defined by 
$$
\|F\|_E=\|\frac{F}{E}\|_{L^2(\R)},\qquad F\in\HH(E).
$$
If $E\in H\!B$, then $\Theta=E^*/E$ is a meromorphic inner function in $\C_+$, meaning that $\Theta$ is an inner function and that $\Theta$ has an analytic continuation to an open neighborhood of $\C_{+}^{-}$.  Conversely,  each meromorphic inner function $\Theta$ admits a representation $\Theta=E^*/E$ for some entire function $E\in HB$. One can see from the identity $\mathscr{K}_{U} = \mathscr{H}^2 \cap U \overline{\mathscr{H}^2}$ that when $\Theta=E^*/E$, the operator 
$F\mapsto F/E$ is unitary from $\HH(E)$ onto the model space $\mathscr{K}_{\Theta}$, that is to say, 
\begin{equation}\label{eq:deBR-model}
\mathscr{K}_{\Theta}=\frac{1}{E}\HH(E).
\end{equation} 
When $E(z)=e^{-i\pi z}$, one can check that $E \in HB$, $\Theta = E^{*}/E$ satisfies $\Theta(z)= e^{2i\pi z}$, and 
$
\mathscr{K}_{\Theta}=e^{i\pi z}\HH(E)=e^{i\pi z}PW.
$

\begin{proof}[Proof of Theorem \ref{nnxbbxvxrx55}]
Fix $\delta \in (0, \tfrac{1}{4})$ and set
$$
E_\delta(z)=(z+i)\prod_{k = 1}^{\infty}\left(1-\frac{z}{k-\delta-ik^{-4\delta}}\right)\left(1-\frac{z}{-k+\delta-ik^{-4\delta}}\right).
$$
 It is shown in 
\cite{Lyubarski-Seip} that $E_{\delta} \in HB$,  
\begin{equation}\label{eq1:exemple-Lyu-Seip}
\HH(E_\delta)=PW,
\end{equation}
with equivalent norms, 
and 
\begin{equation}\label{eq2:exemple-Lyu-Seip}
|E_\delta(x)|\simeq (1+|x|)^{2\delta} \dist(x,\Lambda_\delta),\qquad x\in\R,
\end{equation}
where 
\begin{align*}
\Lambda_{\delta} & = E_{\delta}^{-1}(\{0\}) \\
& = \{k - \delta - i k^{-4\delta}: k \geqslant 1\} \cup \{-k + \delta - i k^{-4\delta}: k \geqslant 1\} \cup \{-i\}.
\end{align*}
If we define $I_\delta=E_\delta^*/E_\delta$, then $I_\delta$ is a meromorphic inner function on $\C_+$. Define $\Psi_\delta(z)=e^{i\pi z}E_\delta(z)$ and use \eqref{eq:deBR-model} and \eqref{eq1:exemple-Lyu-Seip} to obtain
$$
\Psi_\delta \mathscr{K}_{I_\delta}=e^{i\pi z} E_\delta \mathscr{K}_{I_\delta}=e^{i\pi z} \HH(E_{\delta})=e^{i\pi z} PW=\mathscr{K}_{\Theta},
$$
where $\Theta(z) = e^{2 \pi i z}$.
Hence $\Psi_{\delta}$ is a multiplier from $\mathscr{K}_{I_\delta}$ onto $\mathscr{K}_{\Theta}$. We now argue  that $\Psi_{\delta}$ is unbounded. Indeed, the zero set $\Lambda_{\delta}$ of $E_{\delta}$ contains  
$$z_k = (k - \delta) - i k^{-4 \delta}, \quad k \geqslant 1.$$
For each interval $(k - \delta, k + 1 - \delta)$, the zeros $z_k$ and $z_{k + 1}$ lie just below the respective endpoints $k - \delta$ and $k + 1 - \delta$. If $x_k$ is 
the midpoint of $(k - \delta, k + 1 - \delta)$, one can see that 
$\mbox{dist}(x_k, \Lambda_{\delta}) \geqslant \tfrac{1}{2}.$
From \eqref{eq2:exemple-Lyu-Seip} we conclude that  
$$|E_{\delta}(x_k)| \simeq (1 + x_k)^{2 \delta} \mbox{dist}(x_k, \Lambda_{\delta}) \gtrsim (1 + x_k)^{2 \delta} \simeq k^{2 \delta} $$
which goes to infinity as $k \to \infty$. The fact that $\Psi_{\delta}$ is unbounded now follows.
\end{proof}

 This example can be transferred to the disk via
$u = I_{\delta} \circ \omega^{-1}, v = \Theta \circ \omega^{-1}, \phi = \Psi_{\delta} \circ \omega^{-1}$, and applying Lemma \ref{L2mult}.

\subsection*{Crofoot once again}

Crofoot proved that $\phi \K_u = \K_v \implies \MM(u, v) = \C \phi$. A natural question to ask is whether or not $\MM(u, v) = \C \phi \implies \phi \K_u = \K_v$? The answer, in general,  is no. Similar to Theorem 6.9, we construct our example in the upper-half plane setting. 

\begin{Theorem}\label{39782194uro3iq}
There are two inner functions $B$ and $\Theta$ on $\C_{+}$  such that $\MM(B, \Theta) = \C \Psi$, with $\Psi \not \equiv 0$, but $\Psi \mathscr{K}_{B} \subsetneq \mathscr{K}_{\Theta}$. 

\end{Theorem}

\begin{proof}
Let $\Theta(z)=e^{i2\pi z}$, so that $\mathscr{K}_{\Theta}=e^{i\pi z}PW$, and let $E(z)$ be the canonical product associated
with the sequence $\Lambda=\{-i+n+\operatorname{sign}(n)\delta\}_{n\in\Z}$ where we now choose the limit
case in the Ingham-Kadets theorem: $\delta=1/4$. As before, set  $B=E^*/E$.

It is known that the family $\mathcal{F}=\{e^{i\lambda_nx}:n\in\Z\}$ is complete and minimal in $L^2(-\pi,\pi)$  \cite[p.~178]{Levin}, from which it can also be deduced that  $E$ is of exponential type $\pi$  (see some standard computations in \cite[p.~177]{Levin} along with a more general result \cite[Theorem 1]{Yu}). 
This yields the following two properties: (i) $\mathcal{H}(E)\subset PW$; (ii) $\ker T_{\overline{\Theta}B}=\{0\}$.
To see (i), observe first 
that on $\R$ we have $E(x)\simeq (1+|x|)^{-2\delta}= (1+|x|)^{-1/2}$ \cite[p.178]{Levin}
so that when $f\in \mathcal{H}(E)$ (see \eqref{o48tugjflkds}), then 
\[
 \int_{\R}\frac{|f|^2}{|E|^2}dm\simeq\int_{\R}|f|^2(1+|x|)dm<\infty,
\]
implying 
that $f \in L^2(\R)$. Moreover, since $E$ is of exponential type $\pi$, if $f\in \mathcal{H}(E)$, then $f$ is also of exponential type $\pi$. So, by an alternate definition of the Paley-Wiener space, we conclude that $f\in PW$.
Property (ii)  follows from the completeness of $\mathcal{F}$ which means that $\Lambda$ is a uniqueness sequence for $PW$. This is equivalent to  $\ker T_{\overline{\Theta}B}=\{0\}$.

We are now in a position to prove our claim.
By (i), as in the proof of Theorem 6.9, define $\Psi(z)=e^{i\pi z}E(z)$ and use \eqref{eq:deBR-model} and \eqref{eq1:exemple-Lyu-Seip} to obtain
$
\Psi \mathscr{K}_{B}=e^{i\pi z} E\mathscr{K}_{B}=e^{i\pi z} \HH(E)\subset e^{i\pi z} PW=\mathscr{K}_{\Theta},
$
and so $\Psi\in \mathcal{M}(B,\Theta)$.
By Corollary 6.3, the dimension of the multiplier space is bounded by that of $\ker T_{\overline{b_i^+\Theta}B}$. By (ii), we have $\ker T_{\overline{\Theta}B}=\{0\}$. Now
$T_{\overline{b_i^+\Theta}B}=T_{\overline{b_i^+}}T_{\overline{\Theta}B}$, and $\dim\ker T_{\overline{b_i^+}}=1$, so, by injectivity of $T_{\overline{\Theta}B}$, at most one function can be sent to 0 by $T_{\overline{b_i^+\Theta}B}$. So the multiplier algebra is at most one dimensional, and, since $\varphi$ already belongs to this algebra, its dimension is precisely one.
Finally it is clear that the weight $(1+|x|)$ appearing in the norm of $\mathcal{H}(E)$ does not produce an equivalent norm to that in $PW$ (one could for instance consider the family $f_n(z)=\frac{\dst\sin(\pi(z-n))}{\dst\pi(z-n)}$) so that $\mathcal{H}(E)\subsetneq PW$.
\end{proof}

\subsection*{Multipliers and Ahern-Clark points} 
When $\MM(u, v) \not = \{0\}$ we know from Proposition \ref{23094ygtoriefb} that $\sigma(u) \subset \sigma(v)$. Is it the case that the boundary behavior in $\K_u$ is the same as in $\K_v$?
To discuss this further, we need the following result of Ahern and Clark \cite{AC70}: For an inner function $u$, 
every $f \in \K_u$ has a non-tangential limit  at $\zeta$ if and only if 
$$\varliminf_{z \to \zeta} \frac{1 - |u(z)|}{1 - |z|} <  \infty.$$
The last equivalent condition says that $u$ has a finite {\em angular derivative} at $\zeta$ and $\zeta$ is called an {\em Ahern-Clark point} for $\K_u$.

In the upper-half plane case note that $\infty$ is an Ahern-Clark point for a model space $\mathscr{K}_{U}$ precisely when $U \circ \omega^{-1}$ has a finite angular derivative at $z = 1$ (equivalently $U$ has an angular derivative at $\infty$). When $U$ is a Blaschke product with zeros $\mu_n$, this happens precisely when 
\begin{equation}\label{77we98re87r8e7re87}
\sum_{n \geqslant 1} \Im \mu_{n} < \infty.
\end{equation}

\begin{Proposition}\label{oooooAC}
There exists two inner functions $U$ and $V$ in the upper half plane such that $\MM(U,V)$ is
non trivial, $\sigma(U)=\sigma(V)=\{\infty\}$, and $V$ has an angular derivative  at $\infty$ while $U$ does not.
\end{Proposition}

\begin{proof}
Let 
$$E_{1}(z) = \prod_{n = 1}^{\infty} (1 + \frac{z}{2^n i}), \quad E_{2}(z) = \prod_{n = 1}^{\infty} (1 - \frac{z}{2^n - 2^{-2n}i}).$$
Standard estimates from canonical products yield
$$\left|\frac{E_1(z)}{E_{2}(z)}\right| \asymp \left|\frac{z + 2^{m} i}{z - 2^{m} + 2^{-2m} i}\right|, \quad |z| \in [2^{m} - 2^{m - 2}, 2^{m} + 2^{m - 1}].$$
Observe that this fraction is largest when $z$ is close to $2^{m}$ where it behaves like $2^{3 m}$. Setting 
$\widetilde{E}_{2} := (z + \frac{i}{2})^3 E_2,$ we get that 
 $E_1/\widetilde{E}_2$ is bounded on $\C_{+}$ and for any $F \in \mathscr{H}(E_1)$ we have 
$$\frac{F}{\widetilde{E}_{2}} = \frac{F}{E_1} \cdot \frac{E_1}{\widetilde{E}_2}.$$Thus $F \in \mathscr{H}(\widetilde{E}_{2})$. Hence $E_1/\widetilde{E}_2$ is a multiplier from $\mathscr{K}_{U}$ to $\mathscr{K}_{V}$ for the inner functions 
$U = E_{1}^{*}/E_1$ and $V = \widetilde{E}_{2}^{*}/\widetilde{E}_{2}$. The assertions about the Ahern-Clark properties follow from \eqref{77we98re87r8e7re87}. 
\end{proof}

\section{Multipliers and Clark measures}\label{M-Clark-Meas}

If $u$ is inner, we can associate \cite[p.~3]{Duren} a unique positive finite measure $\sigma_u$ on $\T$, called the {\it Clark measure},  via the identity 
\begin{equation}\label{weq78we9dshdjk}
\frac{1-|u(z)|^2}{|1-u(z)|^2}=\int_\T \frac{1-|z|^2}{|z-\xi|^2}\,d\sigma_u(\xi),\quad z\in\D.
\end{equation}
Note that $\sigma_u \perp m$ and that $u(0) = 0$ if and only if $\sigma_u$ is a probability measure. This process can be reversed \cite{CRM, MR2198367}.

We now exploit these measures to obtain additional information about multipliers. 
Using straightforward arguments from the theory of reproducing kernel Hilbert spaces, one obtains the following. 
\begin{Lemma}\label{Lem:CNS-Multiplier}
Let $u,v$ be two inner functions and $\varphi \in H^2$. Then $\varphi\in\MM(u,v)$  if and only if there exists a bounded linear operator $L_{\phi}: \K_v \to \K_u$ satisfying 
$
L_\varphi(k_\lambda^v)=\overline{\varphi(\lambda)}k_\lambda^u, \lambda\in\D.
$

\end{Lemma}



Here is the rephrasing of the lemma above in terms of Clark measures. 

\begin{Theorem}\label{Thm:CNS-abstract-Multiplier}
Let $u,v$ be two inner functions and  $\sigma_u,\sigma_v$ be their associated Clark measures. For $\varphi \in H^2$, the following are equivalent:
\begin{enumerate}
\item[(i)] $\phi \in \MM(u, v)$;
\item[(ii)] there exists a bounded linear operator $\mathfrak L_\varphi: L^2(\sigma_v) \to L^2(\sigma_u)$ satisfying 
\begin{equation}\label{ppsd9sdbbvvxx}
\mathfrak L_\varphi(k_\lambda)=\overline{\varphi(\lambda)}\frac{1-\overline{u(\lambda)}}{1-\overline{v(\lambda)}} k_\lambda,\qquad \lambda\in\D.
\end{equation}
\end{enumerate}
\end{Theorem}

\begin{proof}
Assume that $\varphi\in\MM(u,v)$. By Lemma~\ref{Lem:CNS-Multiplier}, the (bounded) operator $L_{\varphi}: \K_v \to \K_u$ satisfies 
$L_{\varphi} k_\lambda^v=\overline{\varphi(\lambda)}k_\lambda^u,\lambda\in\D$.
Define $\mathfrak L_\varphi :=V_u^{-1} L_{\varphi} V_v: L^2(\sigma_u) \to L^2(\sigma_v)$, where the Clark operator $V_{u}: L^2(\sigma_{u}) \to \K_u$ 
is defined by 
$
V_u k_\lambda=(1-\overline{u(\lambda)})^{-1}k_\lambda^u, \lambda\in\D.
$
A result of Poltoratski \cite{MR1223178} says that every $f\in\K_u$ has radial limits $\sigma_u$-almost everywhere and
$V_u ^{-1}(f)=f$
on the carrier of $\sigma_u$. The identity in \eqref{ppsd9sdbbvvxx} now follows. 

It is easy to see that the above argument can be reversed.
\end{proof}

\begin{Remark}
A similar criterion for multipliers of de Branges--Rovnyak spaces $\mathscr H(b)$ appears in \cite{MR1152354}. 
\end{Remark}

\begin{Corollary}\label{Cor:absolute-continuity}
Let $u,v$ be inner with associated Clark measures $\sigma_u$ and $\sigma_v$ satisfying $\sigma_u\ll \sigma_v$. If $\phi = (1 - v)/(1 - u)$ and $h=d\sigma_u/d\sigma_v$, the following are equivalent: (i) $\varphi \in\MM(u,v)$; (ii) $h\in L^\infty(\sigma_v)$. 
\end{Corollary}

\begin{proof}
$(ii) \implies (i)$: Using Theorem~\ref{Thm:CNS-abstract-Multiplier}, $\phi \in \MM(u, v)$ if and only if there exists a bounded linear operator $\mathfrak L_\varphi:L^2(\sigma_v)\longrightarrow L^2(\sigma_u)$ such that 
$$
\mathfrak L_\varphi(k_\lambda)=\overline{\varphi(\lambda)}\frac{1-\overline{u(\lambda)}}{1-\overline{v(\lambda)}} k_\lambda=k_\lambda,\qquad \lambda\in\D.
$$
For every $f\in L^2(\sigma_v)$, we have 
\begin{align*}
\int_\T |f(\xi)|^2\,d\sigma_u(\xi) & =\int_\T |f(\xi)|^2 h(\xi)\,d\sigma_v(\xi) \leq \|h\|_{L^\infty(\sigma_v)} \|f\|^2_{L^2(\sigma_v)}.
\end{align*}
Hence if we define  $\mathfrak L_\varphi(f)=f$ for $f\in L^2(\sigma_v)$, then $\mathfrak L_\varphi$ is bounded from $L^2(\sigma_v)$ into $L^2(\sigma_u)$, which proves $(1-v)/(1-u) \in \MM(u,v)$. 

$(i)\Longrightarrow (ii)$: Again using Theorem~\ref{Thm:CNS-abstract-Multiplier}, the map $\mathfrak L_\varphi(k_\lambda)=k_\lambda$ extends linearly to a bounded operator from $L^2(\sigma_v)$ into $L^2(\sigma_u)$. In particular, for any $f$ in the linear span of $\{k_\lambda:\lambda\in\D\}$, we have
$$
\int_\T |f|^2 h\,d\sigma_v=\int_\T |f|^2\,d\sigma_u\lesssim \int_\T |f|^2 \,d\sigma_v.
$$
Since the linear span of $\{k_\lambda:\lambda\in\D\}$ is dense in $L^2(\sigma_v)$ (use $\sigma_v \perp m$ along with \cite[p.~59]{Garnett}), we get $h\in L^\infty(\sigma_v)$. 
\end{proof}

\begin{Remark}
It was shown in \cite{Sa} that if $\sigma_u\ll \sigma_v$ and  $h:=d\sigma_u/d\sigma_v$, then $h \in L^2(\sigma_v)$ if and only if  $(1-v)/(1-u)\in H^2$. 

\end{Remark}

\begin{Example}
If $v(z)=\exp((z+1)/(z-1))$, one can show \cite[p.~235]{MSGMR} that the Clark measure $\sigma_v$ is discrete and given by 
$$
\sigma_v=\sum_{n = -\infty}^{\infty}c_n\delta_{z_n},  \quad z_{n} = \frac{2 \pi in - 1}{2 \pi i n + 1}, \quad c_n =\frac{2}{4\pi^2 n^2+1}.$$
Now pick $c_{n}'$ satisfying $0 \leq c_{n}' \leq M c_n$ for some $M \geqslant 1$ and define 
$
\mu'=\sum_{n \geqslant 1}c_{n}'\delta_{z_n}.
$
See \cite[Ch.~11]{MSGMR} for the details on this. 
In other words, we have $d\mu'= h d\sigma_v$, where $0 \leq h \leq M$. Then there is a unique inner function $u$ such that its associated Clark measure is precisely $\mu'$. Corollary~\ref{Cor:absolute-continuity} says that 
$(1 - v)/(1 - u) \in \MM(u, v)$. This construction can be done more generally starting from any finite measure $\sum_{n \geqslant 1} c_n \delta_{z_n}$ on $\T$ and its associated inner function $v$. See also \cite{GMR}.
\end{Example}


\bibliographystyle{plain}

\bibliography{references}

\end{document}